\documentclass[a4paper,11pt]{article}
\usepackage[T1]{fontenc}
\usepackage[utf8]{inputenc}
\usepackage{lmodern}
\usepackage[english]{babel}
\usepackage{amsmath}
\usepackage{color}
\usepackage{amsfonts}
\usepackage{amssymb}
\usepackage{amsthm}
\newtheorem{lm}{Lemma}
\newtheorem{prop}{Property}
\newtheorem{thm}{Theorem}
\newtheorem{defi}{Definition}
\newcommand{\ddiff}{\mathbf{d}}
\newcommand{\ima}{\mathrm{i}}

\title{Construction of a Free L\'evy Process as high-dimensional limit of a Brownian Motion on the Unitary Group}
\author{Micha\"el Ulrich}
\author{Micha\"el Ulrich \\\\
  \multicolumn{1}{p{.7\textwidth}}{\centering\emph{Laboratoire de Mathématiques de Besançon\\ Université de Franche-Comté\\ 16, Route de Gray\\ 25 030 Besançon Cedex, France}}\\\\
   \multicolumn{1}{p{.7\textwidth}}{\centering\emph{Institut für Mathematik und Informatik, Ernst-Moritz-Arndt Universität Greifswald, 17487 Greifswald, Germany}}
  }

\begin{document}

\maketitle
\tableofcontents

\begin{abstract}
It is well known that freeness appears in the high-dimensional limit of independence for matrices. Thus, for instance, the additive free Brownian motion can be seen as the limit of the Brownian motion on hermitian matrices. More generally, it is quite natural to try to build free Lévy processes as high-dimensional limits of classical matricial Lévy processes.\newline
We will focus here on one specific such construction, discussing and generalizing the work done previously by Biane in \cite{Bia}, who has shown that the (classical) Brownian motion on the Unitary group $U\left(d\right)$ converges to the free multiplicative Brownian motion when $d$ goes to infinity. We shall first recall that result and give an alternative proof for it. We shall then see how this proof can be adapted in a more general context in order to get a free Lévy process on the dual group (in the sense of Voiculescu) $U\langle n\rangle$. This result will actually amount to a truly noncommutative limit theorem for classical random variables, of the which Biane's result constitutes the case $n=1$.
\end{abstract}

\section{Biane's result about the Brownian motion on the Unitary group}
In all the following, we assume that a unital noncommutative probability space $\left(A,\phi\right)$ be given. Let us remind what we mean by that definition: a unital noncommutative probability space is a couple $(A,\phi)$ where $A$ is a unital $*$-algebra and $\phi$ is a linear functional on $A$ such that $\phi(a^*a)\geq0$ for each $a\in A$ and $\phi(1)=1$.\newline
We will also write by $\delta_{ab}$ Kronecker's symbol, which is equal to $0$ when $a\neq b$ and is equal to $1$ when $a=b$. 
Let us recall following definitions and result:
\begin{defi}
We denote by $(\nu_t)_{t\geq0}$ the same family of measures on the unit circle as in \cite{Bia}, ie $\nu_t$ is the only probability measure such that $\xi_{\nu_t}(z)=z\exp[\frac{1}{2}\frac{1+z}{1-z}]$, where $\xi_{\nu_t}$ is the inverse function of $\frac{\psi_{\nu_t}}{1+\psi_{\nu_t}}$ and $\psi_{\nu_t}=\int\frac{z\zeta}{1-z\zeta}d\nu_t(\zeta)$ where the integration is done on the unit circle.
\end{defi}

\begin{defi}
A free multiplicative Brownian motion is a family $\left(U_t\right)_{t\geq 0}$ such that:
\begin{itemize}
\item For every $0\leq t_1<t_2<\ldots <t_n$, the family $\left(U_{t_1}, U_{t_2}U_{t_1}^{-1},\ldots,U_{t_n}U_{t_{n-1}}\right)$ is free.
\item For every $0\leq s<t$ the element $U_tU_s^{-1}$ has a distribution $\nu_{t-s}$.  
\end{itemize}
\end{defi}
In his paper \cite{Bia}, Biane proved that Brownian motion on the group $U\left(d\right)$ converges, as $d$ goes to infinity, towards a multiplicative free Brownian motion. To do this, he proves first the convergence of the marginals using representation theory arguments and secondly the freeness of the increments. We suggest here that there is an other way to prove the convergence of the marginals based on the Itô formula.\newline
Let us first observe that the Brownian motion on the Unitary group $U\left(d\right)$ can be defined as the unique solution of:
$$\mathbf{d}U_t^{\left(d\right)} = \mathrm{i}\ddiff H_tU_t^{\left(d\right)}-\frac{1}{2}U_t^{\left(d\right)}\ddiff t$$
 with initial condition $U_0=I$. Note that we denote by $\mathrm{i}$ the complex number, so as to differentiate it from the index $i$. In the same way we write $\mathbf{d}$ the differential operator so as to distinguish it from the size of the matrices.
 In this equation, we have noted by $H_t$ a Brownian motion on hermitian matrices defined by:
 \begin{itemize}
   \item The family $(H_{ij}(t))_{1\leq i\leq j\leq d}$ is an independent family of random variables
   \item For $1\leq i\leq d$, we have $H_{ii}(t)$ a gaussian variable $\mathcal{N}(0,\frac{1}{d})$
   \item For $1\leq k\leq j\leq d$, we have $H_{kj}(t)= H^{(1)}_{kj}(t)+\mathrm{i}H^{(2)}_{kj}(t)$ with $H^{(1)}_{kj}(t)$ and $H^{(2)}_{kj}(t)$ two independent gaussian variables $\mathcal{N}(0,\frac{1}{2d})$
   \item The matrix $H(t)$ is hermitian for each $t$. 
 \end{itemize}
 In particular this means that each entry of $H_t$ is of variance $1/d$.\newline
 
 Note: we shall omit the exponent $\left(d\right)$ when there is no confusion possible.\newline
Let us now denote by $f_{k_1,\ldots,k_r}$ the following function of $t$: $$f_{k_1,\ldots,k_r}=\mathbb{E}\left[tr\left(U_t^{k_1}\right)\ldots tr\left(U_t^{k_r}\right)\right]$$ where the trace is normalized\footnote{The convention we adopt in this paper is following: whenever we mean the normalized trace, we write $tr$ and we write $Tr$ whenever we speak of the usual trace.} by $1/d$. We will find a differential equation involving those functions. 
\begin{lm} \label{lmfondamental}
We have the following formula:
\begin{eqnarray*}
\ddiff \left(U_{i_1j_1}\ldots U_{i_rj_r}\right)&=&\text{ martingale }-\frac{1}{2}\sum_{k=1}^rU_{i_1j_1}\ldots U_{i_rj_r}\ddiff t\\&-&\frac{\ddiff t}{d}\sum_{1\leq p<q\leq r}U_{i_1j_1}\ldots U_{i_pj_q}\ldots U_{i_qj_p}\ldots U_{i_rj_r} 
\end{eqnarray*}
This means that the non-martingale part is constituted by two terms, the first one where nothing is changed in the indices and the second one where you have switched two indices: $j_q$ replaces $j_p$ and $j_p$ replaces a $j_q$.
\end{lm}
\begin{proof}
This is obtained by using Itô's formula and by reasoning for each element in the matrix, because:
$$\ddiff \left(U_{i_1j_1}\ldots U_{i_rj_r}\right)=\sum_{k=1}^r U_{i_1j_1}\ldots \left(\ddiff U_{i_kj_k}\right)\ldots U_{i_rj_r}+\sum_{1\leq k<l\leq r}\prod_{s\neq k,l}U_{i_sj_s}\ddiff\left[U_{i_kj_k},U_{i_lj_l}\right]$$
The $[.,.]$ denotes the quadratic variation. We remark that:
$$\forall{i,j}, \ddiff U_{ij}(t)=\mathrm{i}\sum_{r=1}^d\ddiff H_{ir}U_{rj}-\frac{1}{2}U_{ij}\ddiff t$$
and
\begin{eqnarray*}
\ddiff[H_{i_kr_k},H_{i_lr_l}]&=&\ddiff[H_{i_kr_k}^{(1)}+\ima H_{i_kr_k}^{(2)},H_{i_lr_l}^{(1)}+\ima H_{i_lr_l}^{(2)}]\\
\end{eqnarray*}
But we know that the quadratic variation of two processes is zero if they are independent. Thus, $\ddiff[H_{i_kr_k},H_{i_lr_l}]$ is equal to:
\begin{itemize}
\item If $i_k=i_l$ and $j_l=j_k$, $\ddiff[H_{i_kj_k},H_{i_lj_l}]=\frac{1}{2d}-\frac{1}{2d}=0$
\item If $i_k=j_l$ and $j_k=i_l$, $\ddiff[H_{i_kj_k},H_{i_lj_l}]=\frac{1}{2d}+\frac{1}{2d}=\frac{1}{d}$
\item And it is equal to zero in all other cases.
\end{itemize}
And thus, the quadratic variation can be expressed as:
\begin{eqnarray*}
  \ddiff[U_{i_kj_k},U_{i_lj_l}]&=&\ima\sum_{r_l,r_k=1}^d U_{r_kj_k}U_{r_lj_l}\ddiff[H_{i_kr_k},H_{i_lr_l}]+\text{martingale}\\
  &=&\ima U_{i_lj_k}U_{i_kj_l}
\end{eqnarray*}

\end{proof}

When we take the expectation, the martingale part vanishes.\newline
If we expand $f_{k_1,\ldots,k_r}$, we get:
\begin{eqnarray*}
  f_{k_1,\ldots,k_r}&=&\frac{1}{d^r}\mathbb{E}[\sum_{\substack{i_1^1,\ldots,i_{k_1}^1\\\ldots\\i_1^r,\ldots,i_{k_r}^r}=1}^d U_{i_1^1i_2^1}\ldots U_{i_{k_1}^1i_1^1}\ldots U_{i_1^ri_2^r}\ldots U_{i_{k_r}^ri_1^r}]
\end{eqnarray*} 
To get a system of differential equations we will use the former formula that we have obtained thanks to Itô's Lemma. Especially we must see how the last term, switching $p$ and $q$, can be rewritten in terms of the functions $f_{k_1,\ldots,k_r}$. There are actually two cases to study: first when $p$ and $q$ come from the same trace and second when they come from different traces.\newline
\textbf{When they come from the same trace:} If for instance $p$ and $q$ both come from the $m$\textsuperscript{th} trace, the contribution of this trace is of the kind:
$$\frac{1}{d^r}\ldots U_{i_1^mi_2^m}\ldots U_{i_p^mi_{p+1}^m}\ldots U_{i_q^mi_{q+1}^m}\ldots U_{i_{k_m}^mi_1^m}\ldots$$
 
So when we do the switching it yields:
$$\frac{1}{d^r}\ldots U_{i_1^mi_2^m}\ldots U_{i_p^mi_{q+1}^m}\ldots U_{i_q^mi_{p+1}^m} U_{i_{q+1}^mi_{q+2}^m}\ldots U_{i_{k_m}^mi_1^m}$$
And when we sum over all those indices we see that we actually get: $df_{k_1,\ldots,k_m-\left(q-p\right),q-p,\ldots,k_r}$, ie the switching has produced one more trace.\newline
\textbf{When they come from two different traces:} We shall here suppose that $p$ comes from the $u$\textsuperscript{th} trace and $q$ comes from the $v$\textsuperscript{th} trace, with $u<v$. The contribution of those two traces are:
$$\frac{1}{d^r}\ldots U_{i_1^ui_2^u}\ldots U_{i_p^ui_{p+1}^u}\ldots U_{i_{k_u}^ui_1^u}\ldots U_{i_1^vi_2^v}\ldots U_{i_q^vi_{q+1}^v}\ldots U_{i_{k_v}^vi_1^v}\ldots$$
Switching $p$ and $q$ yields to:
$$\frac{1}{d^r}\ldots U_{i_1^ui_2^u}\ldots U_{i_p^ui_{q+1}^v}\ldots U_{i_{k_u}^ui_1^u}\ldots U_{i_1^vi_2^v}\ldots U_{i_q^vi_{p+1}^u} U_{i_{q+1}i_{q+2}}\ldots U_{i_{k_v}^vi_1^v}\ldots$$
And so if we sum over all indices we see that we get $\frac{1}{d}f_{k_1,\ldots k_u+k_v,\ldots k_r}$, ie we have merged two traces together.\newline

So, if we put it all together we see by using Lemma \ref{lmfondamental}  that the system of differential equations we get is:
\begin{eqnarray*}
  f^\prime_{k_1,\ldots,k_r}&=&-\frac{k_1+\ldots k_r}{2}f_{k_1,\ldots,k_r}-\sum_{\kappa=1}^r\sum_{l=1}^{k_\kappa}\left(k_\kappa-l\right)f_{k_1,\ldots,k_\kappa-l,l,\ldots,k_r}\\
  &-&\frac{1}{d^2}\sum_{1\leq\kappa<\lambda\leq r}\sum_{p=1}^{k_\kappa}\sum_{q=1}^{k_\lambda}f_{k_1,\ldots,k_\kappa+k_\lambda,\ldots}
\end{eqnarray*}

Let us observe here that we have a nice combinatorial structure for these equations. Indeed, we can interpret $(k_1,\ldots,k_r)$ as an integer partition for the integer $k_1+\ldots+k_r$. By doing so, we see that the equation only involves partitions for the same integer because we either split an integer into two parts or we merge two integers into one. These equations thus have the same structure as the equations in Proposition $2.3$ in \cite{TLevy} via the identification between a permutation and the length of the cycles of its canonical decomposition.\newline
Let us also note that an integer $l$ has only finitely many partitions.\footnote{Without going into the details of the theory of integer partitions, we may find a gross upper bound for this number in the following way: A partition of $l$ cannot have more than $l$ parts. So let's consider a line consisting of $l+l-1=2l-1$ boxes. We then put crosses in $l-1$ boxes. Each such cross helps separate two parts of the partition. For instance:
\setlength{\unitlength}{0.5cm}
\begin{picture}(8,2)(0,0)
  \put(0, 0){\line(1,0){7}}
  \put(0,0){\line(0,1){1}}
  \put(1,0){\line(0,1){1}}
  \put(2,0){\line(0,1){1}}
  \put(3,0){\line(0,1){1}}
  \put(4,0){\line(0,1){1}}
  \put(5,0){\line(0,1){1}}
  \put(6,0){\line(0,1){1}}
  \put(7,0){\line(0,1){1}}
  \put(1,0){\line(1,1){1}}
  \put(1,1){\line(1,-1){1}}
  \put(2,0){\line(1,1){1}}
  \put(2,1){\line(1,-1){1}}
  \put(4,0){\line(1,1){1}}
  \put(4,1){\line(1,-1){1}}
\end{picture} 
represents the partition $(1,1,2)$ of the integer $4$. Hence we see that the number of such partitions is bounded by ${ {2l-1}\choose{l-1}}$, which is finite.}
So that means that each function is involved in a system of finitely many linear differential equations with fixed initial conditions. 
\newline
What can we say about the convergence of this family of functions? We actually have that for each $r\geq1$ and every $k_1,\ldots,k_r\geq1$, the function $f^{\left(d\right)}_{k_1,\ldots,k_r}$ converges, as $d$ goes to infinity, towards a function $f_{k_1,\ldots,k_r}$ verifying:
$$f^\prime_{k_1,\ldots,k_r}=-\frac{k_1+\ldots +k_r}{2}f_{k_1,\ldots,k_r}-\sum_{\kappa=1}^r\sum_{l=1}^{k_\kappa}\left(k_\kappa-l\right)f_{k_1,\ldots,k_\kappa-l,l,\ldots,k_r}$$
Indeed, let us fix such a partition $k_1+\ldots+k_r=k$. If we note $P(k):=\left\{(k_1,\ldots,k_r)|r\geq0,k_1+\ldots+k_r=k\right\}$ the set of partitions of the integer $k$, we have just shown that this set if finite. The function $f^{\left(d\right)}_{k_1,\ldots,k_r}$ thus only shows up in a finite number of linear differential equations with constant coefficients. This finite number of differential equations can be rewritten in a matricial form: let $\varPhi_t^{(d)}$ be a vector in $\mathbb{C}^{\sharp P(k)}$ consisting of all functions $f^{\left(d\right)}_{p_1,\ldots,p_l}$ where $p_1,\ldots,p_l$ is a partition of the same integer $k$. Then $\varPhi^{(d)}$ is solution of a differential equation of the form: $$(\varPhi^{(d)})^\prime=A^{(d)}\varPhi^{(d)}$$
where $A^{(d)}$ is a (constant) matrix formed with the coefficients of our differential equations. It is well-known that $\varPhi^{(d)}$ is thus of the form $\varPhi^{(d)}=\varPhi^{(d)}_0e^{A^{(d)}t}$. But the coefficients of the equations for $f^{(d)}$, namely $A^{(d)}$, converge towards the coefficients for the equation of $f$, namely $A$,  and thus $\varPhi^{(d)}$ converges towards $\varPhi$, or in other words, $f^{\left(d\right)}_{k_1,\ldots,k_r}$ converges towards $f_{k_1,\ldots,k_r}$.

We will now denote by $F_{k_1,\ldots,k_r}$ the function $\phi\left(u_t^{k_1}\right)\ldots\phi\left(u_t^{k_r}\right)$ where $u$ is here a free multiplicative Brownian motion. To prove the convergence of the marginals it will be enough to prove that the family of functions $F$ verifies the differential equations system:
$$F^\prime_{k_1,\ldots,k_r}=-\frac{k_1+\ldots k_r}{2}F_{k_1,\ldots,k_r}-\sum_{\kappa=1}^r\sum_{l=1}^{k_\kappa}\left(k_\kappa-l\right)F_{k_1,\ldots,k_\kappa-l,l,\ldots,k_r}$$
Indeed, if we have proven it, then it implies that for all $r\geq1$ and all $0\leq t_1\leq\ldots\leq t_r$ the function $f_{t_1,\ldots,t_r}^{\left(d\right)}$ converges towards $F_{t_1\ldots t_r}$ when $d$ goes to infinity. In particular, if we take $r=1$, we see that we have the convergence of the marginals (in moments).\newline
In order to prove that formula we must remark that a free multiplicative Brownian motion is given by a free stochastic equation with initial conditions $u_0=1$ ($1$ is the unit element of $A$):
$$\ddiff u_t=\ima \ddiff X_tu_t-\frac{1}{2}u_t\ddiff t$$
where $X_t$ is a free additive Brownian motion. This result is stated in \cite{Bia}'s Theorem 2. We will simplify the calculations by putting $V_t:=e^{t/2}u_t$. 
Using the free analogue of Itô's Lemma (see e.g. \cite{KS}, Theorem 5), Biane demonstrated following formula
\begin{eqnarray*}
  \ddiff V_t^n&=&\ima\sum_{k=0}^nV_t^k\ddiff X_tV_t^{n-k}-\sum_{k=1}^{n-1}kV_t^k\phi\left(V_t^{n-k}\right)\ddiff t
\end{eqnarray*}
In other words this means:
$$\ddiff u_t^n=\ima \sum_{k=0}^nu_t^k\ddiff X_tu_t^{n-k}-\sum_{k=1}^{n-1}ku_t^k\phi\left(u_t^{n-k}\right)\ddiff t-\frac{n}{2}u_t^n\ddiff t$$
Taking the trace of it we obtain:
$$\phi\left(u_t^n\right)^\prime=-\sum_{k=1}^{n-1}k\phi\left(u_t^k\right)\phi\left(u_t^{n-k}\right)-\frac{n}{2}\phi\left(u_t^n\right)$$
And so it finally yields the following system of differential equations:
\begin{eqnarray*}
  F_{k_1,\ldots k_r}^\prime&=&-\frac{k_1+\ldots+k_r}{2}F_{k_1\ldots k_r}-\sum_{\kappa=1}^r\sum_{p=1}^{k_\kappa-1}pF_{k_1,\ldots,p,k_\kappa-p,\ldots,k_r}
\end{eqnarray*}
And this is exactly the system we wanted because $F_{k_1,\ldots,p,k_\kappa-p,\ldots}=F_{k_1,\ldots,k_\kappa-p,p,\ldots}$. \newline
To put it in a nutshell: we were able to reprove Biane's result by using a different method (by comparing systems of differential equations) to prove the convergence of marginals. The freeness of the increments can still be proven as did Biane but it will also follow from the results of section $4$. We will now try to use that alternative method to generalize Biane's result. To do that we will need the concept of dual groups.

\section{Dual groups in the sense of Voiculescu and Lévy processes}
We will here briefly introduce dual groups as they were first defined by Voiculescu. For more information on this subject one can read \cite{Vo}. In the sequel we denote by $\sqcup$ the free product of unital $*$-algebras.
\begin{defi}[Dual semigroups]
A (unital) dual semigroup is a triple $\left(B,\Delta,\delta\right)$ where $B$ is a $*$-algebra and $\Delta:B\rightarrow B\sqcup B$ and $\delta:B\rightarrow\mathbb{C}$ are $*$-homomorphisms such that
$$\left(\Delta\bigsqcup id_B\right)\circ\Delta=\left(Id_B\bigsqcup \Delta\right)$$
$$\left(\delta\bigsqcup id_B\right)\circ\Delta=id_B=\left(Id_B\bigsqcup\delta\right)\circ\Delta$$
The former property is called coassociativity, whereas the latter is the counit property.\newline
When considering the free product $B\bigsqcup B$, in order to differentiate between elements coming from the $B$ on the left and elements coming from the $B$ on the right, we will talk about the left and the right legs of $B\bigsqcup B$.
\end{defi}
We shall be in this paper particularly interested in one dual group:
\begin{defi}[Unitary Dual Group]
For $n\geq1$, we call Unitary Dual Group the dual group $\left(U\langle n\rangle,\Delta,\delta\right)$ defined by:
\begin{itemize}
\item The $*$-algebra $U\langle n\rangle$ is generated by $n^2$ generators $\left(u_{ij}\right)_{1\leq i,j\leq n}$ verifying the relations:
$$\forall{1\leq i,j\leq n}\text{ }\sum_{k=1}^nu_{ki}^*u_{kj}=\delta_{ij}=\sum_{k=1}^nu_{ik}u_{jk}^*$$
\item The coproduct is given by:
$$\Delta u_{ij}=\sum_{k}u_{ik}^{(1)}u_{kj}^{(2)}$$
where the exponent $(1)$ (resp. $(2)$) indicates that the element is taken from the left (resp. right) leg of $U\langle n\rangle\sqcup U\langle n\rangle$.
\item The counit is given by: $\delta u_{ij}=\delta_{ij}$. 
\end{itemize}
\end{defi}
Dual semigroups are particularly useful to define free Lévy processes in the most general case.
\begin{defi}[Lévy processes]
We shall assume that we have a dual semigroup $\left(B,\Delta,\delta\right)$ and some unital noncommutative probability space $\left(A,\phi\right)$. A free (resp. tensor independent) Lévy process on the semigroup $B$ over the noncommutative probability space $(A,\phi)$ is a family $\left(j_{s,t}\right)_{0\leq s\leq t}$ of $*$-homomorphisms from $B$ to $A$ such that:
\begin{itemize}
\item (Increment Property) For every $0\leq s\leq t\leq r$ we have:
$$\left(j_{st}\sqcup j_{tr}\right)\circ\Delta=j_{sr}$$
\item (Stationarity) We have for every $0\leq s\leq t$: $j_{0,t-s}=j_{s,t}$
\item (Freeness of the Increments) For every $0\leq t_1\leq t_2\leq\ldots\leq t_{n+1}$, the increments $j_{t_1t_2},\ldots,j_{t_nt_{n+1}}$ are free (resp. tensor independent).
\item (Weak continuity) For each $b\in B$ and each $s\geq0$, we have: $\lim_{t\rightarrow s^+}\phi\circ j_{s,t}(b)=\delta(b)$
\end{itemize}
\end{defi}
How can these concepts be applied in our case? We could generalize Biane's question by taking $U^{\langle d\rangle}_t$ a Brownian motion on the Unitary Group $U(nd)$, where $n$ is a fixed integer. The matrix $U_t^{\langle d\rangle}$ can be decomposed in $n^2$ blocks of size $d\times d$. In the sequel of the article we will denote by $[U_t^{\langle d\rangle}]_{ij}$ the $(i,j)$\textsuperscript{th} block of our Brownian motion. For each $d$ we thus get a quantum stochastic process on the Dual Unitary Group by setting for $0\leq s\leq t$:
\begin{eqnarray*}
  j_{st}^{\langle d\rangle)}&:&U\langle n\rangle\rightarrow (A,\phi)\\
  &&u_{ij}\mapsto [U_t^{\langle d\rangle}]_{ij}
\end{eqnarray*}
We will in the sequel of the article omit the exponent $\langle d\rangle$ whenever no confusion can arise.\newline
The question that is natural to ask and that generalizes Biane's result is whether or not $j_{st}$ converges to a Lévy process on $U\langle n\rangle$ in the limit when $d$ goes to infinity.\newline
We will show that we have following result
\begin{thm}[Main Theorem]\label{maintheorem}
We assume that $\phi$ is tracial.\newline
Let $X=(X_{ij})_{1\leq i,j\leq n}$ be a matrix whose entries are free stochastic variables verifying that:
\begin{itemize}
\item For each $i$, $X_{ii}$ is an additive free Brownian motion.
\item For every $i\neq j$, $X_{ij}=X_{ij}^{(1)}+\ima X_{ij}^{(2)}$ with $\sqrt{2}X_{ij}^{(1)}$ and $\sqrt{2}X_{ij}^{(2)}$ who are two additive free Brownian motions who are free one with another. 
\item For each $i,j$ we have $X_{ij}=X_{ji}^*$.
\item The family $(X_{ij})_{1\leq i\leq j\leq n}$ is free.
\end{itemize}
Let also $\Psi=(\Psi_{ij})$ be a free stochastic process defined by the free stochastic equation with initial condition $\Psi_0=I$: 
$$d\Psi_t=\frac{\ima}{\sqrt{n}} dX_t\Psi_t-\frac{1}{2}\Psi_tdt$$
Through $\Psi$ we may define a free Lévy process $J$ through\footnote{By calculating $d(\sum_k \Psi_{ki}^*\Psi_{kj})$ we find zero. Moreover, when we calculate $d(\sum_k\Psi_{ik}\Psi^*_{jk})$ we find a free stochastic diffenrential equation that is verified by the constant $\delta_{ij}$. By unicity of the solution (see e.g. \cite{KS}[Theorem $4$]), we have that $\sum_k\Psi_{ik}\Psi^*_{jk}=\delta_{ij}$. Thus $J_{st}$ respects the defining relations of $U\langle n\rangle$}:
\begin{eqnarray*}
  J_{st}&:&U\langle n\rangle\rightarrow (A,\phi)\\
  &&u_{ij}\mapsto \Psi_{ij}
\end{eqnarray*}
Then, $(j_{st}^{\langle d\rangle)})$ converges towards $(J_{st})$ as $d$ goes to infinity.
\end{thm}

\section{Convergence of the marginals}
We will first study the convergence of the marginals. Hence we will fix in this section a $t\geq 0$. To prove such a convergence we must study the moments of the type $\phi\circ j_{0t}(u_{i_1j_1}^{\epsilon_1}\ldots u_{i_rj_r}^{\epsilon_r})$, where $\epsilon_1,\ldots,\epsilon_r\in\left\{\varnothing,*\right\}$. For convenience, we will identify $\varnothing$ with $0$ and $*$ with $1$. We will use exactly the same method as in the first section but, because there are $n^2$ variables, we will have many more indices.\newline

\subsection{Notations}
We consider the dual group $U\langle n\rangle$ which is generated by $n^2$ variables. We will need to introduce some notations to describe all the indices that will be involved.\newline
From now on and until the end of the paper, when we have a matrix $M\in\mathcal{M}_{nd}(\mathbb{C})$, we will denote:
\begin{itemize}
\item by $M_{ij}$ the $(i,j)$-matrix entry of $M$.
\item by $[M]_{ij}$ the $(i,j)$-block of size $d\times d$ of the matrix $M$
\end{itemize}
We denote by $[\mathcal{I}]$ the set $[\mathcal{I}]=\left\{1,\ldots,n\right\}^2\times \left\{0,1\right\}$. For such a triple $\alpha=(i,j,\epsilon)$, we will denote $[U]_\alpha$ the $d\times d$ block $[U]^\epsilon_{ij}$ where we identity $\epsilon=1$ with $*$ and $\epsilon=0$ with $\varnothing$.\newline
 We denote by $\mathcal{I}$ the set $\mathcal{I}=\left\{1,\ldots,nd\right\}^2\times \left\{0,1\right\}$. For such a triple $\rho=(\mu,\nu,\omega)$, we will denote $U_\rho$ the coefficient $U_{\mu\nu}$ if $\omega=0$ and the coefficient $\bar{U}_{\mu\nu}$ if $\omega=1$.\newline
When $\Psi$ is in $\mathcal{M}_n(\mathcal{A})$, with $\mathcal{A}$ a $*$-algebra, we denote by $\Psi_{\alpha}$ the element $\Psi_{ij}^\epsilon$.

\subsection{A system of differential equations for the Brownian motion on $U(nd)$}

To achieve our purpose we need to consider the family of functions (as always, we will omit the exponents everytime we may do so without risk):
\begin{eqnarray*}
  &&\gamma^{\langle d\rangle}_{\alpha_{11},\ldots,\alpha_{k_11};\ldots;\alpha_{1r},\ldots,\alpha_{k_rr}}\\
  &=&\mathbb{E}[tr([U]_{\alpha_{11}}\ldots [U]_{\alpha_{k_11}})\ldots tr(\ldots [U]_{\alpha_{k_rr}})]
\end{eqnarray*}
where $r\geq 1; k_1,\ldots,k_r\in\mathbb{N},\alpha_{kl}\in[\mathcal{I}]$.\newline
In other words, we take functions very similar to what we had before in the simpler case of the convergence to Biane's result. They still are the product of traces\footnote{The renormalization is here done with a coefficient $1/d$.}. The difficulty arises here from the fact that we consider blocks and that we thus have to consider all possible products of the blocks and their adjoints. The indices we use specify which $U_{ij}$ appear and if they have a $*$ or not and the semicolumns separate two traces. We will, as previously, try to find a system of differential equations. Let us fix the indices $\alpha_{11}\ldots \alpha_{k_rr}$. \newline
Again, we apply Lemma \ref{lmfondamental} in order to calculate the differential equation. For the sake of simplicity let us first observe what happens if we suppose that there are no $*$ in our function and we will later explain how to get the general case. As previously we treat separately the case where the switch occurs inside a same trace and the case where it affects two distinct traces.\newline
\textbf{The switch occurs in the same trace:} Let's say that the switch is between $p$ and $q$ inside the $\kappa$\textsuperscript{th} trace. Then, when we develop the traces, we see that the contribution of this trace, after the switch, is of the type:
\begin{eqnarray*}
  &&\mathbb{E}[\sum_{s_{11}\ldots s_{k_rr}}\ldots U_{(i_{p\kappa}-1)d+s_{p\kappa},(j_{q\kappa}-1)d+s_{q\kappa}}\ldots U_{(i_{q\kappa}-1)d+s_{q\kappa},(j_{p\kappa}-1)d+s_{p\kappa}}\ldots   ]
\end{eqnarray*}
As we could have expected the $\kappa$\textsuperscript{th} trace will be divided into two distinct traces: we get $d\gamma_{\ldots;i_{1\kappa}j_{1\kappa},\ldots,i_{p\kappa}j_{q\kappa},i_{q+1,\kappa}j_{q+1,\kappa},\ldots;i_{p+1,\kappa}j_{p+1,\kappa},\ldots,i_{q\kappa}j_{p\kappa};\ldots}$ (we recall that the normalization constant we now use for the trace is $1/d$).\newline
\textbf{The switch concerns two distinct traces: }If we do the calculations, we see that we reunite these two traces and that we get a multiplicative factor $1/d$.\newline
So, if we put it all together (in the case we have no $*$ at all), the equation we will have is:
\begin{eqnarray*}
  &&\gamma^\prime_{\alpha_{11},\ldots,\alpha_{k_11};\ldots;\ldots,\alpha_{k_rr}}\\
  &=&-\frac{k_1+\ldots+k_r}{2}\gamma_{\alpha_{11},\ldots,\alpha_{k_11};\ldots;\ldots,\alpha_{k_rr}}\\
  &-&\sum_{\kappa=1}^r\sum_{1\leq p<q\leq k_\kappa}\frac{1}{n}\gamma_{\ldots;\alpha_{1\kappa},\ldots,(i_{p\kappa}j_{q\kappa}0),\alpha_{q+1,\kappa},\ldots;\alpha_{p+1,\kappa},\ldots,(i_{q\kappa},j_{p\kappa},0);\ldots}\\
  &+&\mathcal{O}(\frac{1}{d^2})
\end{eqnarray*}

Now, in the general case. We can remark that $[U^*]_{ij}=[U]_{ji}^*$. We also have:
\begin{eqnarray*}
  \ddiff U_{\mu\nu}&=&\ima \sum_{\tau=1}^d\ddiff H_{\mu\tau}U_{\tau\nu}-\frac{1}{2}U_{\mu\nu}\ddiff t\\
  \ddiff\bar{U}_{\mu\nu}&=&-\ima\sum_{\tau=1}^d\bar{U}_{\tau\nu}\ddiff H_{\tau\mu}-\frac{1}{2}\bar{U}_{\mu\nu}\ddiff t
\end{eqnarray*} 
In turn this yields to the more general Lemma:
\begin{lm}
We have, for $\rho_1,\ldots,\rho_r\in\mathcal{I}$:
\begin{eqnarray*}
  \ddiff(U_{\rho_1}\ldots U_{\rho_r})&=&-\frac{r}{2}U_{\rho_1}\ldots U_{\rho_r}\ddiff t\\
  &+&\text{ martingale part}-\frac{\ddiff t}{nd}\sum_{1\leq p<q\leq r}(-1)^{\omega_p+\omega_q}\zeta_{pq}^{\langle d\rangle)}
\end{eqnarray*}
where:
\begin{equation}
\zeta_{pq}^{\langle d\rangle)}=
\begin{cases}
  U_{\rho_1}\ldots U_{\mu_p\nu_q}\ldots U_{\mu_q\nu_p}\ldots U_{\rho_r}&\text{ if }\omega_p=\omega_q=0\\
  U_{\rho_1}\ldots \bar{U}_{\mu_p\nu_q}\ldots \bar{U}_{\mu_q\nu_p}\ldots U_{\rho_r}&\text{ if }\omega_p=\omega_q=1\\
  \sum_{\tau=1}^{nd} \delta_{\mu_p\nu_q}U_{\rho_1}\ldots \bar{U}_{\tau\nu_p}\ldots U_{\tau\nu_q}\ldots U_{\rho_r}&\text{ if }\omega_p=1,\omega_q=0\\
  \sum_{\tau=1}^{nd} \delta_{\mu_p\mu_q}U_{\rho_1}\ldots U_{\tau\nu_p}\ldots \bar{U}_{\tau\nu_q}\ldots U_{\rho_r}&\text{ if }\omega_p=0,\omega_q=1 
\end{cases}
\end{equation}
\end{lm}
\begin{proof}
It is an application of Itô's Lemma along with the observation that:
$$\ddiff[U_{\mu\nu},U_{\theta \eta }]=-\frac{\ddiff t}{nd}U_{\theta\nu}U_{\mu\eta}\text{ and }\ddiff[\bar{U}_{\mu\nu},U_{\theta\eta}]=\sum_{\tau=1}^{nd}\frac{\ddiff t}{nd}B_{\tau\nu}B_{\tau\eta}\delta_{\mu\theta}$$
\end{proof}
So, taking up the same calculations as before, we get following system of differential equations:
\begin{eqnarray*}
  \gamma_{\alpha_{11},\ldots}^\prime&=&-\frac{k_1+\ldots+k_r}{2}\gamma_{\alpha_{11},\ldots}\\
  &-&\sum_{\kappa=1}^r\sum_{1\leq p<q\leq k_\kappa}(-1)^{\epsilon_{p\kappa}+\epsilon_{q\kappa}}\gamma_{(p,q,\kappa)}\\
  &+&\mathcal{O}(\frac{1}{d^2})
\end{eqnarray*}
where we note:\newline
\textbf{If $\epsilon_{p\kappa}=\epsilon_{q\kappa}=0$:}
$$\gamma_{(p,q,\kappa)}=\gamma_{\ldots;\alpha_{1\kappa},\ldots, (i_{p\kappa}j_{q\kappa}\epsilon_{q\kappa}),\alpha_{q+1,\kappa},\ldots;\alpha_{p+1,\kappa},\ldots,(i_{q\kappa}j_{p\kappa}\epsilon_{p\kappa});\ldots}$$
That is, we have a switch exactly as before.\newline
\textbf{If $\epsilon_{p\kappa}=\epsilon_{q\kappa}=1$:}
    $$\gamma_{(p,q,\kappa)}=\gamma_{\ldots;\alpha_{1\kappa},\ldots,\alpha_{p-1,\kappa},(i_{q\kappa}j_{p\kappa}\epsilon_{p\kappa}),\ldots;(i_{p\kappa}j_{q\kappa}\epsilon_{q\kappa}),\ldots;\ldots}$$
    That is, we also have here a switch as we have already seen.\newline 
\textbf{If $\epsilon_{p\kappa}=1,\epsilon_{q\kappa}=0$:}
    $$\gamma_{(p,q,\kappa)}=\sum_{t=1}^n\delta_{i_{p\kappa}i_{q\kappa}}\gamma_{\ldots;\alpha_{1\kappa},\ldots,(tj_{p\kappa}\epsilon_{p\kappa}),(tj_{q\kappa}\epsilon_{q\kappa}),\ldots;\alpha_{p+1,\kappa}\ldots\alpha_{q-1,\kappa};\ldots}$$
The structure is here a little more complicated, with a sum over $t$ and $t$ replacing the indices $i_p$ and $i_q$ and everything situated between the places $p$ and $q$ gets located in a new trace.\newline
    \textbf{If $\epsilon_{p\kappa}=0,\epsilon_{q\kappa}=1$:}
    $$\gamma_{(p,q,\kappa)}=\sum_{t=1}^n\delta_{i_{p\kappa}i_{q\kappa}}\gamma_{\ldots;\alpha_{1\kappa},\ldots,\alpha_{p-1,\kappa},\alpha_{q+1,\kappa},\ldots;(tj_{p\kappa}\epsilon_{p\kappa}),\ldots,(tj_{q\kappa}\epsilon_{q\kappa});\ldots}$$
the structure is almost the same as in the previous case, with the only difference that the places $p$ and $q$ and everything in between gets into a new trace.

\subsection{A system of differential equations for the free stochastic process}

We will now introduce:
$$\Gamma_{\alpha_{11},\ldots;\ldots;\alpha_{1r},\ldots,\alpha_{k_rr}}=\phi(\Psi_{\alpha_{11}}\ldots)\ldots\phi(\Psi_{\alpha_{1r}}\ldots\Psi_{\alpha_{k_rr}})$$

To prove the convergence of the marginals, we will show that $\Gamma$ verifies the system of differential equations that we have just found, in the limit where $d$ goes to infinity.\newline 
By using free stochastic calculus we can see that the quadratic variation is $\ddiff X_{ij}dX_{kl}=\delta_{il}\delta_{jk}\ddiff t$. Moreover, the free stochastic differential equation yields, coefficient by coefficient:
$$\ddiff\Psi_{uv}=\frac{\ima}{\sqrt{n}}\sum_{k=1}^n\ddiff X_{uk}\Psi_{kv}-\frac{1}{2}\Psi_{uv}\ddiff t$$
and
$$\ddiff\Psi_{uv}^*=-\frac{\ima}{\sqrt{n}}\sum_{k=1}^n\Psi_{kv}^*\ddiff X_{ku}-\frac{1}{2}\Psi_{uv}^*\ddiff t$$
This allows us to prove following technical Lemma:
\begin{lm}\label{lmfondfree}
For each $r\geq2$ and all indices we have:
\begin{eqnarray*}
  \ddiff(\Psi_{\alpha_1}\ldots\Psi_{\alpha_r})&=&-\frac{r\ddiff t}{2}\Psi_{\alpha_1}\ldots\Psi_{\alpha_r}\\
  &+&\frac{\ima}{\sqrt{n}}\sum_{l=1}^r\sum_{k=1}^n(-1)^{\epsilon_l}\Psi_{\alpha_1}\ldots
  \left\{
  \begin{array}{r l}
  \ddiff X_{i_lk}\Psi_{kj_l}&\text{ if }\epsilon_l=0\\
  \Psi_{kj_l}^*\ddiff X_{ki_l}&\text{ if }\epsilon_l=1
  \end{array}
  \right\}
  \ldots \Psi_{\alpha_r}\\
  &-&\frac{\ddiff t}{n}\sum_{1\leq p<q\leq r}(-1)^{\epsilon_p+\epsilon_q}\zeta_{pq}
\end{eqnarray*}
where
\begin{eqnarray*}
  \zeta_{pq}=
  \begin{cases}
    \Psi_{\alpha_1}\ldots \Psi_{\alpha_{p-1}}\phi(\Psi_{i_qj_p}^{\epsilon_p}\ldots\Psi_{\alpha_{q-1}})\Psi_{i_pj_q}^{\epsilon_q}\ldots&\text{ if }\epsilon_p=\epsilon_q=0\\
    \Psi_{\alpha_1}\ldots \Psi_{\alpha_{p-1}}\Psi_{i_qj_p}^{\epsilon_p}\phi(\Psi_{\alpha_{p+1}}\ldots\Psi_{\alpha_{q-1}}\Psi_{i_pj_q}^{\epsilon_q})\Psi_{\alpha_{q+1}}\ldots&\text{ if }\epsilon_p=\epsilon_q=1\\
    \sum_{k=1}^n\delta_{i_pi_q}\Psi_{\alpha_1}\ldots \Psi_{\alpha_{p-1}}\phi(\Psi_{kj_p}^{\epsilon_p}\ldots\Psi_{\alpha_{q-1}}\Psi_{kj_q}^{\epsilon_q})\ldots&\text{ if }\epsilon_p=0,\epsilon_q=1\\
    \sum_{k=1}^n\delta_{i_pi_q}\Psi_{\alpha_1}\ldots \Psi_{kj_{p}}^{\epsilon_{p}}\phi(\Psi_{\alpha_{p+1}}\ldots\Psi_{\alpha_{q-1}})\Psi_{kj_q}^{\epsilon_q}\ldots&\text{ if }\epsilon_p=1,\epsilon_q=0
  \end{cases}
\end{eqnarray*}
\end{lm}
\begin{proof}
The proof is done by recurrence and by using Itô's formula. For simplicity's sake we will do it only in the case where all $\epsilon$ are put equal to zero.\newline
For $r=2$ we get:
$$\ddiff(\Psi_{ij}\Psi_{kl})=\frac{\ima}{\sqrt{n}}\sum_{s=1}^n\Psi_{ij}\ddiff X_{ks}\Psi_{sl}+\frac{\ima}{\sqrt{n}}\sum_{s=1}^n\ddiff X_{is}\Psi_{sj}\Psi_{kl}-\Psi_{ij}\Psi_{kl}\ddiff t-\frac{\ddiff t}{n}\phi(\Psi_{kj})\psi_{il}$$
Hence we have the desired result for $r=2$. Let us now assume that the Lemma is right until a certain $r$. Then, by Itô's Lemma:
\begin{eqnarray*}
  \ddiff(\Psi_{u_1v_1}\ldots\Psi_{u_{r+1}v_{r+1}})&=&-\frac{r+1}{2}\psi_{u_1v_1}\ldots\Psi_{u_{r+1}v_{r+1}}\ddiff t\\
  &+&\frac{\ima}{\sqrt{n}}\sum_{k=1}^n\psi_{u_1v_1}\ldots\Psi_{u_rv_r}\ddiff X_{u_{r+1}k}\Psi_{kv_{r+1}}\\
  &+&\frac{\ima}{\sqrt{n}}\sum_{k=1}^n\sum_{l=1}^r\Psi_{u_1v_1}\ldots \ddiff X_{u_lk}\Psi_{kv_l}\ldots\Psi_{u_{r+1}v_{r+1}}\\
  &-&\frac{\ddiff t}{n}\sum_{1\leq p<q\leq r}\Psi_{u_1v_1}\ldots\phi(\Psi_{u_qv_p}\ldots)\Psi_{u_pv_q}\ldots\Psi_{u_{r+1}v_{r+1}}\\
  &-&\frac{\ddiff t}{n}\sum_{l=1}^r\Psi_{u_1v_1}\ldots \Psi_{u_{l-1}v_{l-1}}\phi(\Psi_{u_{r+1}v_l}\ldots\Psi_{u_rv_r})\Psi_{u_lv_{r+1}}
\end{eqnarray*}
And so we see that the result is also right for $r+1$.
\end{proof}
We now introduce, as expected, the family of functions:
\begin{eqnarray*}
  \Gamma_{\alpha_{11},\ldots;\ldots;\alpha_{1r}}&=&\phi(\Psi_{\alpha_{11}}\ldots)\ldots\phi(\Psi_{\alpha_{1r}}\ldots)
\end{eqnarray*}
By applying Lemma \ref{lmfondfree} we get:
\begin{eqnarray*}
  \Gamma_{\alpha_{11},\ldots;\ldots;\alpha_{1r},\ldots}^\prime&=&-\frac{k_1+\ldots+k_r}{2}\Gamma_{\alpha_{11},\ldots;\ldots;\alpha_{1r},\ldots}\\
  &-&\frac{1}{n}\sum_{\kappa=1}^r\sum_{1\leq p<q\leq k_\kappa}(-1)^{\epsilon_p+\epsilon_q}\Gamma_{(p,q,\kappa)}
\end{eqnarray*}
where we defined:
\begin{eqnarray*}
  &&\Gamma_{(p,q,\kappa)}=\\
  &&\begin{cases}
    \Gamma_{\ldots;\alpha_{1\kappa},\ldots,i_{p\kappa}j_{q\kappa}\epsilon_{q\kappa},\ldots,\alpha_{k_\kappa\kappa};i_{q\kappa}j_{p\kappa}\epsilon_{p\kappa},\ldots,\alpha_{q-1,\kappa};\ldots}&\text{ if }\epsilon_{p\kappa}=\epsilon_{q\kappa}=0\\
    \Gamma_{\ldots;\alpha_{1\kappa},\ldots,i_{q\kappa}j_{p\kappa}\epsilon_{p\kappa},\alpha_{q+1,\kappa},\ldots;\alpha_{p+1,\kappa},\ldots,i_{p\kappa}j_{q\kappa}\epsilon_{q\kappa};\ldots}  &\text{ if }\epsilon_{p\kappa}=\epsilon_{q\kappa}=1\\
    \sum_{l=1}^n\delta_{i_{p\kappa}i_{q\kappa}}\Gamma_{\ldots;\alpha_{1\kappa},\ldots, \alpha_{p-1\kappa},\alpha_{q+1,\kappa}\ldots;lj_{p\kappa}\epsilon_{p\kappa},\ldots,lj_{q\kappa}\epsilon_{q\kappa};\ldots}    &\text{ if }\epsilon_{p\kappa}=0,\epsilon_{q\kappa}=1\\
    \sum_{l=1}^n\delta_{i_{p\kappa}i_{q\kappa}}\Gamma_{\ldots;\alpha_{1\kappa},\ldots,lj_{p\kappa}\epsilon_{p\kappa},lj_{q\kappa}\epsilon_{q\kappa},\ldots;\alpha_{p+1,\kappa}\ldots, \alpha_{q-1,\kappa};\ldots} &\text{ if }\epsilon_{p\kappa}=1,\epsilon_{q\kappa}=0\\
  \end{cases}
\end{eqnarray*}
Hence we see that the family of functions $\gamma$ truly converges towards the family of functions $\Gamma$. In particular, taking $r=1$, we see that the $*$-moments of the family $(U_{ij}^{\langle d\rangle)})_{1\leq i,j\leq n}$ converges towards the $*$-moments of $(\Psi_{ij})_{1\leq i,j\leq n}$. This proves the convergence of the marginals.

\section{Conditional expectation}
In order to prove Theorem \ref{maintheorem} we must prove the convergence of all mixed moments of the kind: $\mathbb{E}\circ tr(U_{i_1j_1}^{\epsilon_1}(t_1)\ldots U_{i_rj_r}^{\epsilon_r}(t_r))$ towards $\phi(\Psi_{i_1j_1}^{\epsilon_1}(t_1)\ldots \Psi_{i_rj_r}^{\epsilon_r}(t_r))$. In the previous section we have already proven that this is indeed the case when $\sharp\left\{t_1,\ldots,t_r\right\}=1$. In order to prove the general case we will use a method consisting of computing the joint moments by taking recursively conditional expectations. 

\subsection{Notations}
In order to use this method, we must generalize somewhat our notations. In the sequel, we fix $s\geq0$ and our time variable $t$ will always verify $t\geq s$. 
We note:
\begin{enumerate}
\item by $[\mathcal{I}]$ the set $\left\{1,\ldots,n\right\}^2\times \left\{0,1\right\}\times \mathcal{M}_d^{(s)}$, where $\mathcal{M}_d^{(s)}$ is the set of $d\times d$ matrices whose entries are $\mathcal{F}_s$-measurable random variables. Of course, we have $\mathcal{F}_s=\sigma(j_u,u\leq s)$.
\item by $\mathcal{I}$ the set $\left\{1,\ldots,nd\right\}^2\times \left\{0,1\right\}\times V^{(s)}$, where $V^{(s)}$ designates the set of $\mathcal{F}_s$-measurable random variables.

\item by $\mathcal{I}^f$ the set $\left\{1,\ldots,n\right\}^2\times \left\{0,1\right\}\times \mathcal{A}_s$, where $\mathcal{A}_s$ is the $*$-algebra generated by all $\Psi_{pq}(u),u\leq s$.
\end{enumerate}
We use these sets as sets of indices in the following way:
\begin{enumerate}
\item If $\alpha=(i,j,\epsilon,m)\in[\mathcal{I}]$, we note $[U]_\alpha=m[U]_{ij}^\epsilon$
\item If $\rho=(\mu,\nu,\omega,\pi )\in\mathcal{I}$, we note $U_\rho=\pi U_{\mu\nu}^\omega$
\item If $\alpha=(i,j,\epsilon,m)\in\mathcal{I}^f$, we note $\Psi_\alpha=m\Psi_{ij}^\epsilon$.

\end{enumerate}

\subsection{A system of differential equations for the Brownian motion on $U(nd)$}

We are interested in the family of functions:
\begin{eqnarray*}
  &&\gamma_{\alpha_{11},\ldots,\alpha_{k_11};\ldots;\ldots,\alpha_{k_rr}}(t)\\
  &=&\mathbb{E}[tr([U]_{\alpha_{11}}(t)\ldots [U]_{\alpha_{k_11}}(t))\ldots tr(\ldots [U]_{\alpha_{k_rr}}(t))]
\end{eqnarray*}
In other words, we use the same family as before but we put $\mathcal{F}_s$-measurable elements between the blocks of the Brownian motion. \newline
We want to use the same method as before. We will need following Lemma:
\begin{lm}\label{lmfondespcond}
We have for any choice of indices in $\mathcal{I}$ and for $t\geq s$:
\begin{eqnarray*}
  \ddiff(U_{\rho_1}\ldots U_{\rho_k})&=&-\frac{k}{2}U_{\rho_1}\ldots U_{\rho_k} \ddiff t\\
  &-&\frac{1}{nd}\sum_{1\leq p<q\leq k}(-1)^{\omega_p+\omega_q}\zeta_{pq}^{\langle d\rangle>}\ddiff t\\
  &+&\text{ martingale part}
\end{eqnarray*}
where:
\begin{equation}
\zeta_{pq}^{\langle d\rangle}=
\begin{cases}
  U_{\rho_1}\ldots \pi_pU_{\mu_p\nu_q}\ldots \pi_qU_{\mu_q\nu_p}\ldots U_{\rho_k}&\text{ if }\omega_p=\omega_q=0\\
  U_{\rho_1}\ldots \pi_pU^*_{\mu_p\nu_q}\ldots \pi_qU^*_{\mu_q\nu_p}\ldots U_{\rho_k}&\text{ if }\omega_p=\omega_q=1\\
  \sum_{\tau=1}^{nd} \delta_{\mu_p\mu_q}U_{\rho_1}\ldots \pi_pU^*_{\tau \nu_p}\ldots \pi_qU_{\tau\nu_q}\ldots U_{\rho_k}&\text{ if }\omega_p=1,\omega_q=0\\
  \sum_{\tau=1}^{nd} \delta_{\mu_p\mu_q}U_{\rho_1}\ldots \pi_pU_{\tau\nu_p}\ldots \pi_qU^*_{\tau\nu_q}\ldots U_{\rho_k}&\text{ if }\omega_p=0,\omega_q=1 
\end{cases}
\end{equation}
\end{lm}
\begin{proof}
As always, this is proven using Itô's Lemma.
\end{proof}
Applying this Lemma, we get:

\begin{lm}\label{systequadiff}
The system of differential equations is:
\begin{eqnarray*}
  &&\gamma_{\alpha_{11},\ldots,\alpha_{k_11};\ldots;\ldots,\alpha_{k_rr}}^\prime\\
  &=&-\frac{k_1+\ldots+k_r}{2}\gamma_{\alpha_{11},\ldots,\alpha_{k_11};\ldots;\ldots,\alpha_{k_rr}}\\
  &-&\frac{1}{n}\sum_{\kappa=1}^r\sum_{1\leq p<q\leq k_\kappa}(-1)^{\epsilon_{p\kappa}+\epsilon_{q\kappa}}\gamma_{(p,q,\kappa)}\\
  &+&\mathcal{O}(\frac{1}{d^2})
\end{eqnarray*}
where:\newline
\textbf{If $\epsilon_{p\kappa}=\epsilon_{q\kappa}=0$:}
  $$\gamma_{(p,q,\kappa)}= \gamma_{\ldots;\ldots,(m_{p\kappa},i_{p\kappa}j_{q\kappa}\epsilon_{q\kappa}),\alpha_{q+1,\kappa}\ldots;\alpha_{p+1,\kappa},\ldots,(m_{q\kappa},i_{q\kappa}j_{p\kappa}\epsilon_{p\kappa});\ldots}$$
\textbf{If $\epsilon_{p\kappa}=\epsilon_{q\kappa}=1$:}
    $$\gamma_{(p,q,\kappa)}=\gamma_{\ldots;\ldots,(m_{p\kappa},i_{q\kappa}j_{p\kappa}\epsilon_{p\kappa}),\alpha_{q+1,\kappa},\ldots;\alpha_{p+1,\kappa}, \ldots,(1,i_{p\kappa}j_{q\kappa}\epsilon_{q\kappa});\ldots}$$
\textbf{If $\epsilon_{p\kappa}=1,\epsilon_{q\kappa}=0$:}
    $$\gamma_{(p,q,\kappa)}=\sum_{t=1}^n\delta_{i_{p\kappa}i_{q\kappa}}\gamma_{\ldots;\ldots,(m_{p\kappa},t,j_{p\kappa}\epsilon_{p\kappa}),(t,j_{q\kappa},\epsilon_{q\kappa},1),\ldots;(i_{p+1,\kappa},j_{p+1,\kappa},\epsilon_{p+1,\kappa},m_{q\kappa}m_{p+1,\kappa}),\ldots ;\ldots}$$
\textbf{If $\epsilon_{p\kappa}=0,\epsilon_{q\kappa}=1$:}
   $$ \gamma_{(p,q,\kappa)}=\sum_{t=1}^n\delta_{i_{p\kappa}i_{q\kappa}}\gamma_{\ldots;\ldots,(i_{q+1,\kappa},j_{q+1,\kappa},\epsilon_{q+1,\kappa},m_{p\kappa}m_{q+1,\kappa},)\ldots;(t,j_{p\kappa},\epsilon_{p\kappa},1),\ldots,(t,j_{q\kappa},\epsilon_{q\kappa},m_{q\kappa});\ldots}$$

\end{lm}
The structure is very similar to what we had proved in the previous section. We just have to be careful to what happens with the $m$'s. 
\newline
When we proved Biane's result we saw that the system of differential equations had a combinatorial structure related to the idea of integer partitions. I do not see any obvious combinatorial structure in this generalized formula but it is a question that is worth being asked.

\subsection{A system of differential equations for the free stochastic process}
 Of course, we will be interested in the behavior of the family of functions:
\begin{eqnarray*}
  \Gamma_{\alpha_{11},\ldots,\alpha_{k_11};\ldots}&=&\phi(\Psi_{\alpha_{11}}(t)\ldots)\ldots\phi(\ldots)
\end{eqnarray*}

\begin{lm}
For any choice of indices in $\mathcal{I}^f$ and for $t\geq s$, we have:
\begin{eqnarray*}
  \ddiff(\Psi_{\alpha_1}\ldots\Psi_{\alpha_k})&=&-\frac{k}{2}\Psi_{\alpha_1}\ldots\Psi_{\alpha_k} \ddiff t\\
  &+&\frac{\ima}{\sqrt{n}}\sum_{r=1}^n\sum_{l=1}^k\Psi_{\alpha_1}\ldots\alpha_l
  \left\{\begin{array}{r l}
  \ddiff X_{i_lr}\Psi_{rj_l}&\text{ if }\epsilon_l=0\\
  \Psi_{rj_l}\ddiff X_{ri_l}&\text{ if }\epsilon_l=1
  \end{array}\right\}
  \ldots\Psi_{\alpha_k}\\
  &-&\frac{\ddiff t}{n}\sum_{1\leq p<q\leq k}(-1)^{\epsilon_p+\epsilon_q}\zeta_{pq}
\end{eqnarray*}
where
\begin{eqnarray*}
  \zeta_{pq}=
  \begin{cases}
    \Psi_{\alpha_1}\ldots\phi(\Psi_{i_qj_p}^{\epsilon_p}\ldots\Psi_{\alpha_{q-1}}m_q)\Psi_{i_pj_q}^{\epsilon_q}\ldots&\text{ if }\epsilon_p=\epsilon_q=1\\
    \Psi_{\alpha_1}\ldots \alpha_p\Psi_{i_qj_p}^{\epsilon_p}\phi(\Psi_{\alpha_{p+1}}\ldots\Psi_{i_pj_q}^{\epsilon_q})\Psi_{\alpha_{q+1}}\ldots&\text{ if }\epsilon_p=\epsilon_q=1\\
    \sum_{t=1}^k\delta_{i_pi_q}\Psi_{\alpha_1}\ldots \alpha_p\phi(\Psi_{tj_p}^{\epsilon_p}\ldots \Psi_{tj_q}^{\epsilon_q})\Psi_{\alpha_{q+1}}\ldots&\text{ if }\epsilon_p=0,\epsilon_q=1\\
    \sum_{t=1}^k\delta_{i_pi_q}\Psi_{\alpha_1}\ldots\Psi_{tj_p}^{\epsilon_p}\phi(\Psi_{\alpha_{p+1}}\ldots\alpha_q)\Psi_{tj_q}^{\epsilon_q}\ldots&\text{ if }\epsilon_p=1,\epsilon_q=0
  \end{cases}
\end{eqnarray*}
\end{lm}
\begin{proof}
It is the same proof as before, based on Itô's formula.\end{proof}

Applying this Lemma, we get:

\begin{lm}\label{freesystem}
The system of differential equations for the free stochastic process is:
\begin{eqnarray*}
&&\Gamma_{\alpha_{11},\ldots;\ldots}^\prime\\
&=&-\frac{k_1+\ldots+k_r}{2}\Gamma_{\alpha_{11},\ldots;\ldots}\\
&-&\sum_{\kappa=1}^r\sum_{1\leq p<q\leq k_\kappa}(-1)^{\epsilon_{p\kappa}+\epsilon_{q\kappa}}\Gamma_{(p,q,\kappa)}
\end{eqnarray*}
where:\newline
\textbf{If $\epsilon_{p\kappa}=\epsilon_{q\kappa}=0$:}
  $$\Gamma_{(p,q,\kappa)}=
    \Gamma_{\ldots;\ldots,(i_{p\kappa}j_{q\kappa}\epsilon_{q\kappa}m_{p\kappa}),\ldots;(i_{q\kappa}j_{p\kappa}\epsilon_{p\kappa}m_{q\kappa}),\ldots,\alpha_{q-1,\kappa};\ldots}$$
\textbf{If $(\epsilon_{p\kappa},\epsilon_{q\kappa})=(1,1)$:}
    $$\Gamma_{(p,q,\kappa)}=\Gamma_{\ldots;\ldots,(i_{q\kappa}j_{p\kappa}\epsilon_{p\kappa}m_{p\kappa}),\alpha_{q+1,\kappa},\ldots;\alpha_{p+1,\kappa},\ldots,(i_{p\kappa}j_{q\kappa}\epsilon_{q\kappa}1);\ldots}$$
\textbf{If $\epsilon_{p\kappa}=0,\epsilon_{q\kappa}=1$:}
    $$\Gamma_{(p,q,\kappa)}=\sum_{t=1}^n\delta_{i_{p\kappa}i_{q\kappa}}\Gamma_{\ldots;\ldots,(i_{q+1,\kappa},j_{q+1,\kappa},\epsilon_{q+1,\kappa},m_{p\kappa}m_{q+1,\kappa}),\ldots;(tj_{p\kappa}\epsilon_{p\kappa}1),\ldots,(tj_{q\kappa}\epsilon_{q\kappa}m_{q\kappa});\ldots}$$
\textbf{If $\epsilon_{p\kappa}=1,\epsilon_{q\kappa}=0$:}
    $$\Gamma_{(p,q,\kappa)}=\sum_{t=1}^n\delta_{i_{p\kappa}i_{q\kappa}}\Gamma_{\ldots;\ldots,(tj_{p\kappa}\epsilon_{p\kappa}m_{p\kappa}),(tj_{q\kappa}\epsilon_{q\kappa}1),\ldots;(i_{p+1,\kappa},j_{p+1,\kappa},\epsilon_{p+1,\kappa},m_{q\kappa}m_{p+1,\kappa}),\ldots,\ldots}$$

\end{lm}

\subsection{Recurrence}
We are now able to finish the proof of Theorem \ref{maintheorem}. We want to show that the moments $\mathbb{E}\circ tr([U]_{i_1j_1}(t_1)^{\epsilon_1}\ldots [U]_{i_kj_k}^{\epsilon_k}(t_k))$ converge towards $\phi(\Psi_{i_1j_1}^{\epsilon_1}(t_1)\ldots\Psi_{i_kj_k}^{\epsilon_k}(t_k))$. Let us denote $\sigma=\sharp\left\{t_1,\ldots,t_k\right\}$ the number of different times showing up in our moment. We are going to prove that result through recurrence on $\sigma$.
\begin{enumerate}
  \item If $\sigma=1$ the result has already been shown because it is just the convergence of the marginals.
  \item Let us suppose that the result is true until a certain $\sigma$. We will now consider a moment using $\sigma+1$ different times. We can order those times in increasing order: $t_1\leq t_1\leq\ldots\leq t_{\sigma+1}$.The recurrence hypothesis tells us that: 
  $$(U_{p,q}(t_i))_{\substack{1\leq i\leq\sigma\\1\leq p,q\leq n}}\underset{{\text{in }*\text{-moments}}}{\longrightarrow}(\Psi_{p,q}(t_i))_{\substack{1\leq i\leq\sigma\\1\leq p,q\leq n}}$$
  We can write the moment under consideration as:
  $$\gamma_{(i_1j_1\epsilon_1m_1^{(d)}),,\ldots,(i_kj_k\epsilon_km_k^{(d)})}(t_{\sigma+1})$$
  where the $m_i^{(d)}$ are $\mathcal{F}_{t_\sigma}$-measurable. Now, let us remark that the family of functions $(\gamma_{\alpha_{11},\ldots,\alpha_{k_11};\alpha_{12},\ldots})$ is entirely characterized by the system of differential equations from Lemma \ref{systequadiff} along with all the relationships between the $\{m_{ij}^{(d)},1\leq j\leq r,1\leq i\leq k_j\}$. In the same way, the family $\Gamma_{\ldots}$ is entirely defined by the system from Lemma \ref{freesystem} along with the relationships between the $\{m_{ij},1\leq j\leq r,1\leq i\leq k_j\}$\newline
  Now, the recurrence hypothesis allows us to say that the $m_i^{(d)},1\leq i\leq k$ converges towards some $m_i,1\leq i\leq k$. This tells us that the relationships between the $\{m_i^{(d)}\}$ "converges" towards the relationships between the $\{m_i\}$. Moreover, the system of differential equations from Lemma \ref{systequadiff} converges towards that of Lemma \ref{freesystem}. To put it in a nutshell, this means:
  $$\gamma_{\alpha_1^{(d)},\ldots,\alpha_1^{(d)}}(t_{\sigma+1})\underset{d\rightarrow\infty}{\longrightarrow}\Gamma_{\alpha_1,\ldots,\alpha_k}(t_{\sigma+1})$$
  Or, in other words, we have the convergence of our moment. 
\end{enumerate}
Thus, we have proven that all $*$-moments converge and this means that Theorem \ref{maintheorem} is proven. 

\section{Some examples of calculations and gaussianity}
We will now use the differential equations that we obtained to calculate some simple moments of our process. We will then be able to draw a consequence about the gaussianity of the free process. In the sequel, we denote by $\phi_t$ the function defined on $U\langle n\rangle$ by $\phi_t=\phi\circ J_{0t}$ where $J_t$ is the limit (free) process.
\subsection{The first moments}
Let us take now $1\leq i\neq j\leq n$. We have the following differential equations:
\begin{eqnarray*}
  \frac{d}{dt}\phi_t(u_{ii})&=&-\frac{1}{2}\phi_t(u_{ii})\\
  \frac{d}{dt}\phi_t(u_{ij})&=&-\frac{1}{2}\phi_t(u_{ij})
\end{eqnarray*}
with initial conditions: $\phi_0(u_{ii})=1$ and $\phi_0(u_{ij})=0$. It thus yields:
\begin{eqnarray*}
  \phi_t(u_{ii})&=&e^{-\frac{1}{2}t}\\
  \phi_t(u_{ij})&=&0
\end{eqnarray*}
We find the same expression for $\phi_t(u_{ii}^*)$ and $\phi_t(u_{ij}^*)$ because they obey the same differential equation with the same initial conditions.
\subsection{The second moments}
Let us take $1\leq i,j,k,l\leq n$. We have following equation:
\begin{eqnarray*}
  \frac{d}{dt}\phi_t(u_{ij}u_{kl})&=&-\phi_t(u_{ij}u_{kl})-\phi_t(u_{il})\phi_t(u_{kj})\frac{1}{n}\\
  &=&-\phi_t(u_{ij}u_{kl})-\frac{1}{n}\delta_{il}\delta_{kj}e^{-t}
\end{eqnarray*}
with initial conditions $\phi_0(u_{ij}u_{kl})=\delta_{ij}\delta_{kl}$ because $\Psi_0=I$. This equation is a linear differential equation of order $1$ and the well-known method allows us to say:
\begin{eqnarray*}
  \phi_t(u_{ij}u_{kl})&=&\frac{\delta_{ij}\delta_{kl}}{n}e^{-t}-t\delta_{il}\delta_{kj}e^{-t}
\end{eqnarray*}
The moments $\phi_t(u_{ij}^*u_{kl}^*)$ also obey the same equation with the same initial condition and they therefore have the same expression. If we are interested in $\phi_t(u_{ij}u_{kl}^*)$ we get the equation:
\begin{eqnarray*}
  \frac{d}{dt}\phi_t(u_{ij}u_{kl}^*)&=&-\phi_t(u_{ij}u_{kl}^*)+\frac{1}{n}\sum_{p=1}^n\phi_t(u_{pj}u_{pl}^*)
\end{eqnarray*}
with initial conditions $\phi_0(u_{ij}u_{kl}^*)=\delta_{ij}\delta_{kl}$. This can be put in the form of a system of linear differential equations by puting $\Phi_t=\left(\phi_t(u_{ij}u_{kl})\right)_{1\leq i,j,k,l\leq n}$ seen as a vector of $\mathbb{C}^{n^4}$ and $A=(a_{(r_1,r_2,r_3,r_4),(s_1,s_2,s_3,s_4)})$ as a matrix acting on $\mathbb{C}^{n^4}$, with:
\begin{eqnarray*}\begin{cases}
a_{rs}=0&\text{ if }s_1=s_3\text{ and }r=s\\
a_{rs}=1/n&\text{ if }s_1=s_3\text{ and }r\neq s\\
a_{rs}=-1&\text{ if }r=s\text{ and }r_1\neq r_3
\end{cases}\end{eqnarray*}
  The equation then is:
\begin{eqnarray*}
  \Phi^\prime&=&A\Phi
\end{eqnarray*}
The solution of such an equation is of the form $\Phi_t=Ce^{At}$ with $C$ a constant.
\subsection{Gaussianity}
We would like to define a Brownian motion on $U\langle n\rangle$ as a free stochastic process having the same law (the same $*$-moments) as $\Psi_t$. This would seem natural because it is just the limit of the Brownian motion on $U\left(nd\right)$. To know if this definition makes sense, we would like $\Psi_t$ to verify some properties, and especially the gaussian property as defined in \cite{Fra}, Proposition $1.12$ and in \cite{Sch}, Proposition $5.1.1$. \newline
We define a counit $\delta$ on $U\langle n\rangle$ as the morphism of $*$-algebras verifying $\delta(u_{ij})=\delta_{ij}$. We recall following definition and results from \cite{Fra} and from \cite{Sch}:
\begin{defi}[Definition $1.8$ from \cite{Fra}]\label{gaussian}
Let $\mathcal{B}$ be a unital $*$-algebra equipped with a character $\delta:\mathcal{B}\rightarrow\mathbb{C}$. A Schürmann triple on $(\mathcal{B},\epsilon)$ is a triple $(\pi,\eta,L)$ consisting of:
\begin{itemize}
  \item A unital $*$-representation $\pi:\mathcal{B}\rightarrow L(D)$ on some pre-Hilbert space $D$.
  \item A linear map $\eta:\mathcal{B}\rightarrow D$ verifying:
  $$\eta(ab)=\pi(a)\eta(b)+\eta(a)\epsilon(b)$$
  \item A hermitian linear functional $L:\mathcal{B}\rightarrow\mathbb{C}$ such that:
  $$-\langle\eta(a^*),\eta(b)\rangle=\epsilon(a)L(b)-L(ab)+L(a)\epsilon(b)$$
\end{itemize}
\end{defi}
\begin{prop}[Theorem $1.9$ from \cite{Fra}]
There is a one-to-one correspondence between Schürmann triples, generators of Lévy processes and Lévy processes.
\end{prop}
\begin{defi}[Proposition $5.1.1$ from \cite{Sch}]
We say that a Lévy process on $U\langle n\rangle$ is gaussian if one of the following equivalent properties are verified:
\begin{itemize}
  \item For each $a,b,c\in Ker\delta$, we have $L(abc)=0$.
  \item For each $a,b\in Ker\delta$, we have $L(b^*a^*ab)=0$.
  \item For all $a,b,c\in U\langle n\rangle$ we have the following formula:
  \begin{eqnarray*}
    L(abc)&=&L(ab)\delta(c)+L(ac)\delta(b)+\delta(a)L(bc)-\delta(a)\delta(b)L(c)\\&-&\delta(a)\delta(c)L(b)-L(a)\delta(b)\delta(c)
  \end{eqnarray*}
  \item The representation $\pi$ is zero on $Ker\delta$: $\pi_{|Ker\delta}=0$.
  \item We have for each $a\in U\langle n\rangle$: $\pi(a)=\delta(a) Id$.
  \item For each $a,b$ in $Ker\delta$, we have: $\eta(ab)=0$.
  \item We have for all $a,b$ in $U\langle n\rangle$: $\eta(ab)=\delta(a)\eta(b)+\eta(a)\delta(b)$.
\end{itemize}
\end{defi}

\begin{thm}
Let us take $D=\mathcal{M}_n(\mathbb{C})$. We then define a Schürmann triple by setting:
  $$\eta(u_{jk})=\epsilon_{jk}/\sqrt{n},\eta(u_{jk}^*)=-\epsilon_{kj}\sqrt{n}$$
  $$\pi(u_{jk})=\delta_{jk}Id$$
  $$L(u_{jk})=-\frac{1}{2}\sum_{r=1}^n\langle\eta(u_{rj}^*),\eta(u_{rk})\rangle$$

where $\epsilon_{jk}$ describe the elementary matrices.\newline
Then, the Schürmann triple $(\eta,\pi,L)$ is associated to the Lévy process on $U\langle n\rangle$ we are interested in. 
\end{thm}
\begin{proof}
We prove it by recurrence on the length of the words:\newline
\textbf{For the length $1$:} we have:
\begin{eqnarray*}
  L(u_{jk})&=&-\frac{1}{2}\sum_{r=1}^n\langle \epsilon_{rj},\epsilon_{rk}\rangle=-\frac{1}{2n}\sum_{r=1}^nTr(\epsilon_{jr}\epsilon_{rk})=-\delta_{jk}/2
\end{eqnarray*}
\textbf{Let us suppose the result is true for words of length up to $k$:} We must first find an expression for $\eta$. The cocycle property for $\eta$ allows us to find through an easy recurrence that:
\begin{eqnarray*}
  \eta(u_{i_1j_1}^{\epsilon_1}\ldots u_{i_kj_k}^{\epsilon_k})&=&\sum_{p=1}^k\delta_{i_1j_1}\ldots
  \begin{cases}
    \epsilon_{i_pj_p}&\text{ if }\epsilon_p=0\\
    -\epsilon_{j_pi_p}&\text{ if }\epsilon_p=1
  \end{cases}
  \ldots\delta_{i_kj_k}
\end{eqnarray*}
We can now use the coboundary property to write:
\begin{eqnarray*}
  L(u_{i_1j_1}^{\epsilon_1}\ldots u_{i_{k+1}j_{k+1}}^{\epsilon_{k+1}})&=&\epsilon(u_{i_2j_2}^{\epsilon_1}\ldots u_{i_{k+1}j_{k+1}}^{\epsilon_{k+1}})L(u_{i_1j_1}^{\epsilon_{k+1}})+L(u_{i_2j_2}^{\epsilon_2}\ldots u_{i_{k+1}j_{k+1}}^{\epsilon_{k+1}})\epsilon(u_{i_1j_1}^{\epsilon_1})\\
  &+&\langle\eta(u_{i_1j_1}^{1-\epsilon_1}),\eta(u_{i_2j_2}^{\epsilon_2}\ldots u_{i_kj_k}^{\epsilon_k}u_{i_{k+1}j_{k+1}}^{\epsilon_{k+1}})\rangle\\
  &=&-\frac{k}{2}\delta_{i_1j_1}\ldots\delta_{i_kj_k}\delta_{i_{k+1}j_{k+1}}\\&-&\sum_{2\leq p<q\leq k+1}(-1)^{\epsilon_p+\epsilon_q}\Gamma_{(p,q,1)}\delta_{i_{1}j_{1}}\\
  &-&\delta_{i_1j_1}\delta_{i_2j_2}\ldots\delta_{i_{k+1}j_{k+1}}/2\\
  &+&\clubsuit 
\end{eqnarray*}
where we have used the fact that the Brownian motion on $U(nd)$ at time $t=0$ is just $Id$. So we only have to compute the value of $\clubsuit$, which is the term arising from $\langle\eta(u_{i_1j_1}^{1-\epsilon_1}),\eta(u_{i_2j_2}^{\epsilon_2}\ldots u_{i_kj_k}^{\epsilon_k}u_{i_{k+1}j_{k+1}}^{\epsilon_{k+1}})\rangle$. We also remark that to finish our recurrence, it suffices to show that this $\clubsuit$ is equal to
$$-\sum_{2\leq p \leq k+1}(-1)^{\epsilon_p+\epsilon_{1}}\Gamma_{(1,p,1)}$$
So we may now write:
\begin{eqnarray*}
  &&\langle\eta(u_{i_1j_1}^{1-\epsilon_1}),\eta(u_{i_2j_2}^{\epsilon_2}\ldots u_{i_kj_k}^{\epsilon_k}u_{i_{k+1}j_{k+1}}^{\epsilon_{k+1}})\rangle\\&=&\frac{1}{n}\langle\begin{cases}-\epsilon_{j_1i_1}&\text{ if }\epsilon_1=0\\\epsilon_{i_1j_1}&\text{ if }\epsilon_1=1\end{cases},\sum_{p=2}^{k+1}\delta_{i_2j_2}\ldots
  \begin{cases}
    \epsilon_{i_pj_p}&\text{ if }\epsilon_p=0\\
    -\epsilon_{j_pi_p}&\text{ if }\epsilon_p=1
  \end{cases}
  \ldots\delta_{i_{k+1}j_{k+1}}\rangle\\
  &=&\sum_{p=2}^{k+1}\spadesuit_p
\end{eqnarray*}
We may now study the four cases:\newline
\textbf{Case where $\epsilon_1=\epsilon_p=0$:} we have
\begin{eqnarray*}
  \spadesuit_p&=&-\frac{1}{n}\delta_{i_1j_p}\delta_{i_2j_2}\ldots \delta_{i_pj_1}\ldots=-(-1)^{\epsilon_1+\epsilon_p}\Gamma_{(1,p,1)}
\end{eqnarray*}
\textbf{Case where $\epsilon_1=\epsilon_p=1$:} we have
\begin{eqnarray*}
  \spadesuit_p&=&-\frac{1}{n}\delta_{i_1j_p}\delta_{i_2j_2}\ldots \delta_{i_pj_1}\ldots=-(-1)^{\epsilon_1+\epsilon_p}\Gamma_{(1,p,1)}
\end{eqnarray*}
\textbf{Case where $\epsilon_1=0,=\epsilon_p=1$:} we have
\begin{eqnarray*}
  \spadesuit_p&=&\frac{1}{n}\delta_{i_1i_p}\delta_{i_2j_2}\ldots \delta_{j_pj_1}\ldots=-(-1)^{\epsilon_1+\epsilon_p}\Gamma_{(1,p,1)}
\end{eqnarray*}
\textbf{Case where $\epsilon_1=1,\epsilon_p=0$:} we have
\begin{eqnarray*}
  \spadesuit_p&=&\frac{1}{n}\delta_{i_1i_p}\delta_{i_2j_2}\ldots \delta_{j_pj_1}\ldots=-(-1)^{\epsilon_1+\epsilon_p}\Gamma_{(1,p,1)}
\end{eqnarray*}
Thus, we have proven the result by recurrence.
\end{proof}
\begin{thm}
The Lévy process from Theorem \ref{maintheorem} is gaussian.
\end{thm}
\begin{proof}
It is immediate by using the fifth characterization from Definition \ref{gaussian}.

\end{proof}

Our Lévy process is thus a good candidate to define what we would like to call a Brownian motion on $U\langle n\rangle$.
\section{Conclusion}
We have proven in this article a generalization of Biane's result, namely that the Brownian motion on $U\left(nd\right)$, seen block-wise, converges towards a Lévy process on the Unitary Dual Group $U\langle n\rangle$, as $d$ goes to infinity. Biane's result can thus be seen as a Lévy process on $U\langle 1\rangle$.The proof of our generalized result uses quite elementary tools, ie mainly the convergence of systems of differential equations and combinatorial considerations.\newline
This limit free Lévy process is described by using a free stochastic differential equation whose form is similar to the equation of the Brownian motion on $U(nd)$. A natural question would be to know if other classical matricial Lévy processes arising from (classical) stochastic equations yield (free) Lévy process described by a similar (free) stochastic equation. \newline
Also, this free Lévy process seems to be a good definition for a Brownian motion on our dual group $U\langle n \rangle$.

\nocite{*}
\bibliography{biblio}
\bibliographystyle{plain}

\end{document}